\documentclass[12pt, a4paper]{amsart}

\usepackage{amsmath, amssymb, amsthm, amsfonts}
\usepackage{mathrsfs}
\usepackage{mathtools}
\usepackage[top=2.6cm, bottom=2.5cm, left=2.6cm, right=2.6cm]{geometry}

\theoremstyle{plain}
\newtheorem{thm}{Theorem}
\newtheorem{lem}{Lemma}[section]
\newtheorem{cor}[lem]{Corollary}

\theoremstyle{remark}
\newtheorem{remark}{Remark}

\begin{document}
\makeatletter
\renewcommand{\theequation}{%
\thesection.\arabic{equation}}
\@addtoreset{equation}{section}
\makeatother

\title{The groups $Sp(4n+1)$ and  $Spin(8n-2)$ as framed manifolds} 
\author{Haruo Minami}
\address{H. Minami: Professor Emeritus, Nara University of Education}
\email{hminami@camel.plala.or.jp}
\begin{abstract}
We consider a compact Lie group as a framed manifold equipped with the left invarianat framing $\mathscr{L}$. In a previous paper we have proved that 
the Adams $e_\mathbb{C}$-invariant value of $SU(2n)$ $(n\ge 2)$ gives a generator of the image of $e_\mathbb{C}$ by twisting $\mathscr{L}$ by a certain 
map. In this note we show that in a similar way we can obtain analogous results 
for $Sp(4n+1)$ and $Spin(8n-2)$ $(n\ge 1)$.
\end{abstract}

\maketitle

\section{Introduction and main result}

Let $G$ be a simply-connected compact Lie group of dimension $4l-1$ and of rank $\ge 2$. We regard it as a framed manifold equipped with the left invariant framing $\mathscr{L}$ and write $[G, \mathscr{L}]$ for its bordism class. Then it is well known ~\cite{AS} that 
\[e_\mathbb{C}([G, \mathscr{L}])=0 \] where $e_\mathbb{C} : \pi_{4l-1}^S \to \mathbb{Q}/\mathbb{Z}$ denotes the complex Adams $e$-invariant. In ~\cite{M} we proposed to raise the problem of whether there exists a map
$\lambda : G\to GL(t, \mathbb{R})$ such that the $e_\mathbb{C}$-invariant value of $[G, \mathscr{L}]$ with $\mathscr{L}$ replaced by $\mathscr{L}^\lambda$ gives a generator of the image of $e_\mathbb{C}$, a cyclic subgroup of 
$\mathbb{Q}/\mathbb{Z}$  ~\cite{A}, and verified that this holds true for the case $G=SU(2n)$ $(n\ge 2)$ where $\mathscr{L}^\lambda$ is the  framing obtained by twisting $\mathscr{L}$ by $\lambda$. In this note we consider the cases $G=Sp(4n+1)$ and $G=Spin(8n-2)$ $(n\ge 1)$ where 
$\dim Sp(4n+1)=4(8n^2+5n+1)-1$ and $\dim Spin(8n-2)=4(8n^2-5n+1)-1$.

\begin{thm}  Let $\rho : Sp(4n+1)\to GL(16n+4, \mathbb{R})$ 
be the standard real representation of $Sp(4n+1)$ and 
$\Delta : Spin(8n-2)\to GL(2^{4n-1}, \mathbb{R})$ the spin representation of $Spin(8n-2)$. Then we have
\begin{align*}\tag*{\,(i)}
&e_\mathbb{C}\bigl([Sp(4n+1), \mathscr{L}^{2n\rho}]\bigr)
=(-1)^nB_{8n^2+5n+1}/2(8n^2+5n+1)\qquad (n\ge 1),\\
\tag{ii}
&e_\mathbb{C}\bigl([Spin(8n-2), \mathscr{L}^{(2n-1)\Delta}]\bigr)
=(-1)^nB_{8n^2-5n+1}/2(8n^2-5n+1)\qquad (n\ge 1)
\end{align*}
where $B_l$ denotes the $l$-th Bernoulli number.
\end{thm}

The proof is carried out based on the use of Proposition 2.1 of ~\cite{LS} 
along the same procedure as in the case $G=SU(2n)$. In fact we proceed 
with this by intending to construct a tensor product decomposition of the 
complex line bundle $E$ associated to a certain circle bundle 
$S\to G\to G/S$. This enables us to apply the proposition above to $E$ 
and thereby leads us to the desired conclusion.

\section{Decomposition of $E$ in the symplectic case}

First we consider the symplectic case. 
Let $d(z)=\mathrm{diag}(\bar{z}^{N-1}, z, \ldots, z)\in Sp(N)$ for 
$z\in S^1$ where $S^1$ is the unit circle in $\mathbb{C}$ and 
$\mathrm{diag}(a_1, \ldots, a_N)$ denote the diagonal matrix with diagonal entries $a_i$. Let $S$ be the circle subgroup of $Sp(N)$ generated by $d(z)$. Let $S$ act on $Sp(N)$ by right multiplication. Regard $p : Sp(N)\to Sp(N)/S$ as the principal bundle along with the natural projection. Let $\pi : E=Sp(N)\times_S\mathbb{C}\to Sp(N)/S$ be the canonical line bundle over $Sp(N)/S$ associated to $p$ where $S$ acts on $\mathbb{C}$ as $S^1$, i.e. $d(z)v=zv$ for $v\in \mathbb{C}$. Then it is clear that its unit sphere bundle $S(E)\to Sp(N)/S$ is  naturally isomorphic to $p$ as a principal $S$-bundle. 

Now we look at the case $N=2$ in order to represent the elements of the 2-sphere in analogy the notations in ~\cite{M}. Put
\begin{equation*}
R(rz, u)=\begin{pmatrix} rz & u \\
-\bar{u} & r\bar{z} \end{pmatrix} \in SU(2)\subset Sp(2), \quad
R(sz, jv)=\begin{pmatrix} sz & jv \\
jv & sz \end{pmatrix} \in Sp(2)   
\end{equation*}
where $r, \ s\ge 0, \ z\in S^1,  \ u, \ v\in\mathbb{C}$ and $j$ is the 
quaternion imaginary unit. Then $R(rz, u)d(z)=R(r, uz)$ and 
$R(sz, jv)\,\mathrm{diag}(1, \bar{z}^2)d(z)=R(s, jv\bar{z})$. Here if $r=0$ and $s=0$, then since $|u|=1$ and $|v|=1$, replacing $z'$s by $\bar{u}\bar{z}$ and 
$v\bar{z}$, we have $R(0, uz)d(\bar{u}\bar{z})=R(0, 1)$ and 
$R(0, jv\bar{z})\,\mathrm{diag}(1, (\bar{v}z)^2)d(v\bar{z})=R(0, j)$. 
For brevity we write 
\begin{equation*}
\begin{split}
(r, uz)_R&=R(r, uz) \quad \,\, \text{with} \ uz=1 \ \text{when} \ r=0, \\ 
(s, jv\bar{z})_R&=R(s, jv\bar{z}) \quad \text{with} \ v\bar{z}=1 \ \text{when} \ s=0.
\end{split}
\end{equation*}
Here we assume for later use that either $(r, uz)_R$ with $r=1$ or $(s, jv\bar{z})_R$ with $s=1$ is chosen as the base point of $S^2$ depending on the situation.
Letting $(r, uz)_R$ and $(s, jv\bar{z})_R$ correspond to $(1-2r^2, 2rzu)$ and $(1-2s^2, 2s\bar{z}v)$, respectively, the subspaces of $Sp(2)/S$ consisting of these elements become homeomorphic to $S^2$, respectively. Hence we identify
\[
(r, uz)_R=(1-2r^2, 2ruz), \qquad (s, jv\bar{z})_R=(1-2s^2, 2sv\bar{z}). 
\]
Under these identifications the assignments 
\[R(rz, u)\to (r, uz)_R, \qquad R(sz, jv)\to (s, jv\bar{z})_R\]
along $p : Sp(2)\to Sp(2)/S$, respectively, yield principal $S$-bundles over 
$S^2$ isomorphic to the Hopf bundle $p : S^3\to S^2$. 
With the same notation we write $L$ for the complex line bundles associated to all these principal $S$-bundles constructed in such a manner.

From now, let $G=Sp(4n+1)$ $(n\ge 1)$ and let us aim to show that $E$ can be expressed as isomorphically the tensor product of complex line bundles. 

Let $0\le k \le 4n-1$ and $1\le i \le 4n-k$. Put $n_k=4n-k$. 
For now fix $k$ and put in $G$
\begin{equation*}
R_{i; k}(r_{i; k}z, u_{i; k})=
\left(
\begin{array}{ccccc} 
I_k & 0 & 0 & 0 & 0 \\ 
0 & r_{i; k}z & 0 & u_{i; k} & 0 \\ 
0 & 0 & I_{i-1} & 0 & 0 \\
0 & -\bar{u}_{i; k}& 0 & r_{i; k}\bar{z} & 0\\
0 & 0 & 0 & 0 & I_{n_k-i}
\end{array}
\right)
\end{equation*}
with $R(r_{i; k}z, u_{i; k})\in SU(2)\subset Sp(2)$ and
\begin{equation*}
R_{i; k}(s_{i; k}z, jv_{i; k})=
\left(
\begin{array}{ccccc}  
I_k & 0 & 0 & 0 & 0 \\ 
0 & s_{i; k}z & 0 & jv_{i; k} & 0 \\ 
0 & 0 & I_{i-1} & 0 & 0 \\
0 & \ \, jv_{i; k}& 0 & s_{i; k}z & 0\\
0 & 0 & 0 & 0 & I_{n_k-i}
\end{array} 
\right)
\end{equation*}
with $R(s_{i; k}z, jv_{i; k})\in Sp(2)$ where $I_l$ is the unit matrix of order $l$. Then 
we have  
\begin{equation*}
\textstyle\prod_{i=1}^{n_k}R_{i; k} (r_{i; k}z, u_{i; k})=
\left(
\begin{array}{cccccc} 
I_k & 0 & 0 & 0 & \cdots & 0 \\ 
0 & a_{0, 0 ; k} & a_{0, 1; k} & a_{0, 2; k} &\cdots & a_{0, n_k; k} \\ 
0 & a_{1, 0; k} & a_{1, 1; k}  & a_{1, 2; k} & \cdots & c_{1, n_k; k}  \\
0 & a_{2, 0; k} & 0& a_{2, 2; k} & \cdots & a_{2, n_k; k}  \\
\vdots & \vdots  & \vdots  &\ddots  & \ddots  &   \vdots\\
0 & a_{n_k, 0; k} & 0 & \cdots & 0 & a_{n_k, n_k; k} 
\end{array}
\right)
\end{equation*}
where
\begin{equation*}
\begin{split}
&a_{0, 0; k}=r_{1; k}\cdots r_{n_k; k}z^{n_k},\\
&a_{l, 0; k}=-r_{l+1; k}\cdots r_{n_k; k}\bar{u}_{l; k}z^{n_k-l}
\quad (1\le l\le n_k-1), \hspace{35mm}\\ 
&a_{n_k, 0; k}=-\bar{u}_{n_k; k},\\
&a_{0, 1; k}=u_{1; k},\\
&a_{0, l; k}=r_{1; k}\cdots r_{l-1; k}u_{l; k}z^{l-1} \quad (2\le l\le n_k),\\ 
&a_{l, m; k}=-r_{l+1; k}\cdots r_{m-1; k}\bar{u}_{l; k}u_{m; k}z^{m-l-1}
\quad (1\le l\le m-2, \ 3\le m \le n_k),\hspace{7mm}\\ 
&a_{l, l+1; k}=-\bar{u}_{l; k}u_{l+1; k} \quad (1\le l\le n_k-1), \\
&a_{l, l; k}=r_{l; k}\bar{z} \quad (1\le l\le n_k), \\  
&a_{l, m; k}=0 \quad (l>m).\\
\end{split}
\end{equation*}
and 
\begin{equation*}
\textstyle\prod_{i=1}^{n_k}R_{i; k} (s_{i; k}z, jv_{i; k})=
\left(
\begin{array}{cccccc} 
I_k & 0 & 0 & 0 & \cdots & 0 \\ 
0 & b_{0, 0; k} & jb_{0, 1; k} & jb_{0, 2; k} &\cdots & jb_{0, n_k; k} \\ 
0 & jb_{1, 0; k} & b_{1, 1; k}  & b_{1, 2; k} & \cdots & b_{1, n_k; k}  \\
0 & jb_{2, 0; k} & 0& b_{2, 2; k} & \cdots & b_{2, n_k; k}  \\
\vdots & \vdots  & \vdots  &\ddots  & \ddots  &   \vdots\\
0 & jb_{n_k, 0; k} & 0 & \cdots & 0 & b_{n_k, n_k; k} 
\end{array}
\right)
\end{equation*}
where
\begin{equation*}
\begin{split}
&b_{0, 0; k}=s_{1; k}\cdots s_{n_k; k}z^{n_k},\\
&b_{l, 0; k}=s_{l+1; k}\cdots s_{n_k; k}\bar{v}_{l; k}z^{n_k-l}
\quad (1\le l\le n_k-1), \hspace{35mm}\\ 
&b_{n_k, 0; k}=\bar{v}_{n_k; k},\\
&b_{0, 1; k}=v_{1; k},\\
&b_{0, l; k}=s_{1; k}\cdots s_{l-1; k}v_{l; k}\bar{z}^{l-1} \quad (2\le l\le n_k),\\ 
&b_{l, m; k}=-s_{l+1; k}\cdots s_{m-1; k}\bar{v}_{l; k}v_{m; k}\bar{z}^{m-l-1}
\quad (1\le l\le m-2, \ 3\le m \le n_k),\hspace{7mm}\\ 
&b_{l, l+1; k}=-\bar{v}_{l; k}v_{l+1; k} \quad (1\le l\le n_k-1), \\
&b_{l, l; k}=s_{l; k}z \quad (1\le l\le n_k), \\  
&b_{l, m; k}=0 \quad (l>m).\\
\end{split}
\end{equation*}
From these calculation results we see that if we put 
\begin{equation*}
R_k(r_{i; k}z, u_{i; k})=\textstyle\prod_{i=1}^{n_k}R_{i; k} (r_{i; k}z, u_{i; k}), \quad  
R_k(s_{i; k}z, jv_{i; k})=\textstyle\prod_{i=1}^{n_k}R_{i; k} (s_{i; k}z, jv_{i; k}), 
\end{equation*}
then it follows that  
\begin{equation}
\begin{split}
R_k(r_{i; k}z, u_{i; k})\,
\mathrm{diag}(1, \overset{(k)}\ldots, 1, \bar{z}^{n_k}, z, \ldots, z)
&=R_k(r_{i; k}, u_{i; k}z^i),\\
R_k(s_{i; k}z, jv_{i; k})\,
\mathrm{diag}(1, \overset{(k)}\ldots, 1, \bar{z}^{n_k}, \bar{z}, \ldots, 
\bar{z})&=R_k(s_{i; k}, jv_{i; k}\bar{z}^i).
\end{split}
\end{equation}  
Moreover due to performing the calculation above we know that 
the right hand sides of (2.1), respectively, can be written in the forms:    
\begin{equation}
R_k(r_{i; k}, u_{i; k}z^i)=
\left(
\begin{array}{cccccc} 
I_k & 0 & 0 & 0 & \cdots & 0 \\ 
0 & \underline{a}_{0, 0; k} & \ast & \ast & \cdots & \ast \\ 
0 & \underline{a}_{1, 0; k} & \ast & \ast & \cdots & \ast  \\
0 & \underline{a}_{2, 0; k} & 0 & \ast & \cdots & \ast  \\
\vdots & \vdots  & \vdots  &\ddots  & \ddots  &  \vdots\\
0 & \underline{a}_{n_k, 0; k} & 0 & \cdots & 0 & \ast 
\end{array}
\right)
\end{equation}
with $\ast'$s in $\mathbb{C}$ where
\begin{equation*}
\begin{split}
&\underline{a}_{0, 0; k}= r_{1; k}\cdots r_{n_k; k}, \  
\underline{a}_{l, 0; k}= -r_{l+1; k}\cdots r_{n_k; k}\bar{u}_{l; k}\bar{z}^l  
\quad (1\le l\le n_k-1), \\
&\underline{a}_{n_k, 0; k}= -\bar{u}_{n_k; k}\bar{z}^{n_k}
\end{split}
\end{equation*}
and 
\begin{equation}
R_k(s_{i; k}, jv_{i; k}\bar{z}^i)=
\left(
\begin{array}{cccccc} 
I_k & 0 & 0 & 0 & \cdots & 0 \\ 
0 & \underline{b}_{0, 0; k} & j\underline{b}_{0, 1; k} & j\underline{b}_{0, 2; k} &\cdots & j\underline{b}_{0, n_k; k} \\ 
0 & j\underline{b}_{1, 0; k} &  \underline{b}_{1, 1; k}  & \ast & \cdots & \ast  \\
0 & j\underline{b}_{2, 0; k} & 0& \underline{b}_{2, 2; k} & \cdots & \ast  \\
\vdots & \vdots  & \vdots  &\ddots  & \ddots  &  \vdots\\
0 & j\underline{b}_{n_k, 0; k} & 0 & \cdots & 0 & \underline{b}_{n_k, n_k; k} 
\end{array}
\right)
\end{equation}
with $\ast'$s in $\mathbb{C}$ where
\begin{equation*}
\begin{split}
&\underline{b}_{0, 0; k}=s_{1; k}\cdots s_{n_k; k}, \  
\underline{b}_{l, 0: k}=s_{l+1; k}\cdots s_{n_k; k}v_{l; k}\bar{z}^l \quad (1\le l\le n_k-1), 
\\
&\underline{b}_{n_k, 0; k}=v_{n_k; k}\bar{z}^{n_k},\\
&\underline{b}_{0, 1; k}=v_{1; k}\bar{z}, \ \underline{b}_{0, l; k}
=s_{1; k}\cdots s_{l-1; k}v_{l; k}\bar{z}^l 
\quad (2\le \mathit{l}\le n_k),\\
&\underline{b}_{l, l; k}=s_{i; k}\quad  (1\le l\le n_k).
\end{split}
\end{equation*}

Put $D_0(z)=I_{4n+1}$, $D_k(z)=\mathrm{diag}(z^{4n}, \bar{z}, 
\overset{(k-1)}\ldots, \bar{z}, \bar{z}^{4n+1-k}, 1, \ldots, 1)$ 
$(1\le k\le 4n-1)$. Let $1\le i\le n_k$ and put 
\[
R^{\{i\}}_k(r_{i; k}z, u_{i; k})=\biggl(\textstyle\prod_{l=1}^{n_k}R_{l; k} (r_{l; k}z, u_{l; k})\biggr)D_k(z)
\]
with $r_{l; k}=1$ for all $l, k$ except for $l=i$ and 
\[
R^{\{i\}}_k(s_{i; k}z, jv_{i; k})=\biggl(\textstyle\prod_{l=1}^{n_k}R_{l; k} (s_{l; k}z, jv_{l; k})\biggr)
\mathrm{diag}(1, \overset{(k+1)}\ldots, 1, \bar{z}^2, \ldots, \bar{z}^2)D_k(z) 
\]
with $s_{l; k}=1$ for all $l, k$ except for $l=i$.
Then we know from (2.1) that
\begin{equation}
\begin{split}
R^{\{i\}}_k(r_{i; k}z, u_{i; k})d(z)&=R_{i; k}(r_{i; k}, u_{i; k}z^{i}), \\
R^{\{i\}}_k(s_{i; k}z, jv_{i; k})d(z)&=R_{i; k}(s_{i; k}, jv_{i; k}\bar{z}^{i})
\end{split}
\end{equation}
since $R_{l; k}(r_{l; k}, u_{l; k}z^l)=I_{4n+1}$ and 
$R_{l; k}(s_{l; k}, jv_{l; k}\bar{z}^l)=I_{4n+1}$ when $r_{l; k}=1$ and $s_{l; k}=1$, 
respectively. Now through the elementary matrix calculations we have   
\begin{equation}
\begin{split}
d(x)R_{i; 0}(r_{i; 0}z, u_{i; 0}a^{4n+1})&=R_{i; 0}(r_{i; 0}z, u_{i; 0})d(x),\\
d(x)R_{i; k}(r_{i; k}z, u_{i; k})&=R_{i; k}(r_{i; k}z, u_{i; k})d(x) \quad \text{if} \ k\ge 1, \\
\bar{d}(x)R_{i; 0}(s_{i; 0}z, jv_{i; 0}\bar{x}^{8n})&=R_{i; 0}(s_{i; 0}, jv_{i; 0})\bar{d}(x),
\\
d(x)R_{i; k}(s_{i; k}z, jv_{i; k}\bar{x}^2)&=R_{i; k}(s_{i; k}z, jv_{i; k})d(x) \quad 
\text{if} \ k\ge 1 \quad (x\in S^1),
\end{split}
\end{equation}
where $\bar{d}(x)=\mathrm{diag}(\bar{x}^{4n}, \bar{x}, \ldots, \bar{x}) 
=\mathrm{diag}(1, \bar{x}^2, \ldots, \bar{x}^2)d(x)$.

If we put  
\begin{equation*}
\begin{split}
R_k(r_{i; k}z_i, u_{i; k})&=\textstyle\prod_{i=1}^{n_k}
R^{\{i\}}_k(r_{i; k}z_i, u_{i; k}), \\
R_k(s_{i; k}z_i, jv_{i; k})&=\textstyle\prod_{i=1}^{n_k}
R^{\{i\}}_k(s_{i; k}z_i, jv_{i; k}) \quad (z_i\in S^1),
\end{split}
\end{equation*}
then using the first equation of (2.4) and the former two equations of (2.5) 
we have
\begin{equation*}
\begin{split}
R_0(r_{i; 0}z_i, u_{i; 0})d(z_1)\cdots d(z_{4n})
&=\textstyle\prod_{i=1}^{4n}
R_{i; 0}(r_{i; 0}, u_{i; 0}(z_1\cdots z_{i-1})^{4n+1}z_i^i), \\
R_k(r_{i; k}z_i, u_{i; k})d(z_1)\cdots d(z_{n_k})
&=\textstyle\prod_{i=1}^{n_k}R_{i; k}(r_{i; k}, u_{i; k}z_i^i)\quad 
\text{if} \ k\ge 1
\end{split}
\end{equation*}
where $z_1\cdots z_{i-1}=1$ when $i=0$. Similarly using the other ones 
of (2.4) and (2.5) we have
\begin{equation*}
\begin{split}
R_0(s_{i; 0}z_i, jv_{i; 0})d(z_1)\cdots d(z_{4n})
&=\textstyle\prod_{i=1}^{4n}
R_{i; 0}(s_{i; 0}, jv_{i; 0}(\bar{z}_1\cdots\bar{z}_{i-1})^{8n}\bar{z}_i^i), \\
R_k(s_{i; k}z_i, jv_{i; k})d(z_1)\cdots d(z_{n_k})
&=\textstyle\prod_{i=1}^{n_k}R_{i; k}(r_{i; k}, jv_{i; k}(\bar{z}_1\cdots
\bar{z}_{i-1})^2\bar{z}_i^i)\quad \text{if} \ k\ge 1
\end{split}
\end{equation*}
where $z_1\cdots z_{i-1}=1$ when $i=0$. In the above, for example, if any $r_{i; k}$ is zero, then since $|u_{i; k}|=1$, we can choose $z_{i; k}$ which satisfies 
$u_{i; 0}(z_1\cdots z_{i-1})^{4n+1}z_i^i=1$ or $u_{i; k}z_i^i=1$ depending on whether 
$k=0$ or not. Hence the product component $R_{i; k}(r_{i; k}, w_{i; k})$ with 
$r_{i; k}=0$ on the right side of them can be coverted to $R_{i; k}(0, 1)$ in $G/S$. The same is true in the case of $s_{i; k}=0$ and in fact, 
$R_{i; k}(s_{i; k}, j\omega_{i; k})$ with $s_{i; k}=0$ can be converted to $R_{i; k}(0,  j)$ as an element of $G/S$. 

If we let 
\begin{equation*}
\begin{split}
R(r_{i; k}z_{i; k}, u_{i; k})&=\textstyle\prod_{k=0}^{4n-1}R_k(r_{i; k}z_{i; k}, u_{i; k}), \\
R(s_{i; k}z_{i; k}, jv_{i; k})&=\textstyle\prod_{k=0}^{4n-1}R_k(s_{i; k}z_{i; k}, jv_{i; k}),\\
d(z_{1; k}, \ldots, z_{n_k; k})&
=\textstyle\prod_{k=0}^{4n-1}d(z_{1; k})\cdots d(z_{n_k; k}), 
\end{split}
\end{equation*} 
then observing the procedure used to prove the equations above we finally find that the following equation holds. 
\begin{lem} 
\begin{equation*}
\begin{split}
&R(r_{i; k}x_{i; k}, u_{i; k})R(s_{i; k}y_{i; k}, jv_{i; k})
d(x_{1; k}, \ldots, x_{n_k; k})d(y_{1; k}, \ldots, y_{n_k; k})\\
&=R(r_{i; k}, u_{i; k}m(x_{\le i; k})x^i_{i; k})
R(s_{i; k}, jv_{i; k}m(x_{\le i; n_k})m(y_{\le i; k})\bar{y}^i_{i; k})\quad 
(x_{i; k}, y_{i; k} \in S^1)
\end{split}
\end{equation*}
where $m(z_{\le i; k})$ denotes a monomial of 
$z_{1; k}, \ldots, z_{i; k}$ and where if $r_{i; k}$ and $s_{i; k}$ are zero, then the 
product components $R_{i; k}(r_{i; k}, w_{i; k})$ and $R_{i; k}(s_{i; k}, j\omega_{i; k})$ on the right side can be converted to $R_{i; k}(0, 1)$ and $R_{i; k}(0, j)$, respectively.
\end{lem} 
Putting $(S^2)^\mathit{l}=S^2\times\overset{(l)}\cdots\times S^2$ we define 
a map $\phi : (S^2)^{16n^2+4n} \to G/S$ by
\[
(x_0, \ldots, x_{4n-1}, y_0, \ldots, y_{4n-1})\to 
p(R(r_{i; k}, w_{i; k})R(s_{i; k}, j\omega_{i; k}))
\] 
where $x_k=((r_{1; k}, w_{1; k})_R, \ldots, (r_{n_k; k}, w_{n_k; k})_R)$ and 
$y_k=((s_{1; k}, j\omega_{1; k})_R, \ldots, (s_{n_k; k}, j\omega_{n_k; k})_R)$. 
Then taking into acount Lemma 2.1 we see that
\[P=\{R(r_{i; k}x_{i; k}, u_{i; k})R(s_{i; k}y_{i; k}, jv_{i; k}) | (r_{i; j}, u_{i; j})_R, 
(s_{i; k}, jv_{i; k})_R\in S^2, x_{i; k}, y_{i; k}\in S^1\}\subset G\] 
forms the total space of a principal $S$-bundle endowed with the projection map $q : P\to (S^2)^{16n^2+4n}$ such that $\phi\circ q=p|P$. By putting 
$L^{\boxtimes l}=L\boxtimes\overset{(l)}\cdots\boxtimes L$ we therefore have 
\begin{cor} The induced bundle $\phi^*E$ is isomorphic to $L^{\boxtimes{(16n^2+4n)}}$.  
\end{cor}
Let $(S^2)^\circ$ be the subspace of $S^2$ consisting of $(a, Z)_R$ with $a>0$ and let $((S^2)^\mathit{l})^\circ$ be the direct product of $l$ copies of 
$(S^2)^\circ$. Then we have
\begin{lem} The restriction of 
$\phi$ to $((S^2)^{16n^2+4n})^\circ$ is an injective map.
\end{lem}
\begin{proof} Letting $x=(x_0, \ldots, x_{4n-1})$ and $y=(y_0, \ldots, y_{4n-1})$,  
suppose $\phi(x, y)=\phi(x', y')$, namely $p(R(r_{i; k}, w_{i; k})R(s_{i; k}, j\omega_{i; k}))
=p(R(r'_{i; k}, w'_{i; k})R(s'_{i; k}, j\omega'_{i; k}))$ where we denote by attaching
``\,$'$\," to an element accompanied by $x$ its corresponding element accompanied by $x'$. It is easily seen that this equation can be interpreted as meaning that 
\begin{equation*}\tag{$\ast$}
\begin{split}
&R_0(r_{i; 0}, w_{i; 0})\cdots R_{4n-1}(r_{i; 4n-1}, w_{i; 4n-1})
R_0(s_{i; 0}, j\omega_{i; 0})\cdots R_{4n-1}(s_{i; 4n-1}, j\omega_{i; 4n-1})\\
&=R_0(r'_{i; 0}, w'_{i; 0})\cdots R_{4n-1}(r'_{i; 4n-1}, w'_{i; 4n-1})
R_0(s'_{i; 0}, j\omega'_{i; 0})\cdots R_{4n-1}(s'_{i; 4n-1}, j\omega'_{i; 4n-1})
\end{split}
\end{equation*}
where in this case $r_{i; l}, s_{i; l}>0$ and $r'_{i; l}, s'_{i; l}>0$. 
From (2.2) and (2.3) we know that the product components $R_k(r_{i; k}, w_{i; k})$ and $R_k(s_{i; k}, j\omega_{i; k})$ in the left hand side of $(*)$ are of the form $\mathrm{diag}(I_k, M_k)$, respectively, 
and that the first column of $M_k$ consists of  
\begin{equation*}
\begin{split}
&\underline{a}_{0, 0; k}= r_{1; k}\cdots r_{n_k; k}, \  
\underline{a}_{l, 0; k}= -r_{l+1; k}\cdots r_{n_k; k}\bar{w}_{l; k}  
\quad (1\le l\le n_k-1), \\
&\underline{a}_{n_k, 0; k}= -\bar{w}_{n_k; k}\\
&\underline{b}_{l, l; k}=s_{i; k}\quad  (1\le l\le n_k); \ \text{and}\\
&\underline{b}_{0, 0; k}=s_{1; k}\cdots s_{n_k; k}, \  
\underline{b}_{l, 0: k}=s_{l+1; k}\cdots s_{n_k; k}\omega_{l; k} \quad (1\le l\le n_k-1), \ 
\\
&\underline{b}_{n_k, 0; k}=\omega_{n_k; k},\\
&\underline{b}_{0, 1; k}=\omega_{1; k}, \ 
\underline{b}_{0, l; k}
=s_{1; k}\cdots s_{l-1; k}\omega_{l; k}
\quad (2\le \mathit{l}\le n_k),
\end{split}
\end{equation*}
respectively. The same holds true for those in the right hand side.  

Based on these relations we first prove $x_0=x'_0$. In view of the form of 
$R_k(r_{i; k}, w_{i; k})$ and  $R_k(s_{i; k}, j\omega_{i; k})$ in the case $k=0$ we have 
\begin{equation*}
\underline{a}_{0, 0; 0}\underline{b}_{0, 0; 0}
=\underline{a}'_{0, 0; 0}\underline{b}'_{0, 0; 0}, \ldots, \underline{a}_{l, 0; 0}\underline{b}_{0, 0; 0}
=\underline{a}'_{l, 0; 0}\underline{b}'_{0, 0; 0}, \ldots, 
\underline{a}_{4n, 0; 0}\underline{b}'_{0, 0; 0}
=\underline{a}'_{4n, 0; 0}\underline{b}'_{0, 0; 0}.
\end{equation*}
Here  
$|\underline{a}_{0, 0; 0}|^2+\cdots+|\underline{a}_{4n, 0; 0}|^2=1$ and 
$|\underline{a}'_{0, 0; 0}|^2+\cdots+|\underline{a}'_{4n, 0; 0}|^2=1$. 
Hence we know that $|\underline{b}_{0, 0; 0k}|=|\underline{b}'_{0, 0; 0}|$, i.e. $s_{1; 0}\cdots s_{4n; 0}=s'_{1; 0}\cdots s'_{4n; 0}$. Now since $s_{1; 0}\cdots s_{4n; 0}>0$ we have
\begin{equation*}
\underline{a}_{0, 0; 0}=\underline{a}'_{0, 0; 0}, \ldots, \underline{a}_{l, 0; 0}=\underline{a}'_{l, 0; 0}, \ldots, 
\underline{a}_{4n, 0; 0}=\underline{a}'_{4n, 0; 0}.
\end{equation*}
From the last equation $\underline{a}_{4n, 0; 0}=\underline{a}'_{4n, 0; 0}$ of this  we have $w_{4n; 0}=w'_{4n; 0}$ and so it follows that $r_{4n; 0}=r'_{4n; 0}$. Substituting this into the second equation $\underline{a}_{4n-1, 0; 0}=\underline{a}'_{4n-1, 0; 0}$ 
we have $r_{4n; 0}w_{4n-1; 0}=r'_{4n; 0}w'_{4n-1; 0}$. Here since 
$r_{4n; 0}=r'_{4n; 0}> 0$, $w_{4n-1; 0}=w'_{4n-1; 0}$ and so it also follows that
$r_{4n-1; 0}=r'_{4n-1; 0}$. Repeating this procedure we subsequently obtain $w_{4n-2; 0}=w'_{4n-2; 0}$, $r_{4n-2; 0}=r'_{4n-2; 0}$;   
$w_{4n-3; 0}=w'_{4n-3; 0}$, $r_{4n-3; 0}=r'_{4n-3; 0}$; $\cdots$ ; 
$w_{2; 0}=w'_{2; 0}$, $r_{2; 0}=r'_{2; 0}$ and finally arrive at 
$w_{1; 0}=w'_{1; 0}$, $r_{1; 0}=r'_{1; 0}$. This allows us to conclude that $x_0=x'_0$. 

From the definition it follows that the equation $x_0=x'_0$ is equivalent to 
$R_0(r_{i; 0}, w_{i; 0})=R_0(r'_{i; 0}, w'_{i; 0})$. So we know that these factors can be canceled from both sides of ($\ast$) by multiplying their inverse elements (matrices) and that therefore it can be rewritten as
\begin{equation*}
\begin{split}
&R_1(r_{i; 1}, w_{i; 1})\cdots R_{4n-1}(r_{i; 4n-1}, w_{i; 4n-1})
R_0(s_{i; 0}, j\omega_{i; 0})\cdots R_{4n-1}(s_{i; 4n-1}, j\omega_{i; 4n-1})\\
&=R_1(r'_{i; 1}, w'_{i; 1})\cdots R_{4n-1}(r'_{i; 4n-1}, w'_{i; 4n-1})
R_0(s'_{i; 0}, j\omega'_{i; 0})\cdots R_{4n-1}(s'_{i; 4n-1}, j\omega'_{i; 4n-1}).
\end{split}
\end{equation*}
Similarly by applying the above argument to this equation we have \begin{equation*}
\underline{a}_{0, 0; 1}\underline{b}_{1, 1; 0}
=\underline{a}'_{0, 0; 1}\underline{b}'_{1, 1; 0}, \ldots, \underline{a}_{l, 0; 1}\underline{b}_{1, 1; 0}
=\underline{a}'_{l, 0; 1}\underline{b}'_{1, 1; 0}, \ldots, 
\underline{a}_{4n, 0; 1}\underline{b}'_{1, 1; 0}
=\underline{a}'_{4n, 0; 1}\underline{b}'_{1, 1; 0}.
\end{equation*}
From this in a similar way to the case $k=0$ above we can get $x_1=x'_1$ or equivalently 
$R_1(r_{i; 1}, w_{i; 1})=R_1(r'_{i; 1}, w'_{i; 1})$. Hence the rewritten equation 
above can also be rewritten as
\begin{equation*}
\begin{split}
&R_2(r_{i; 2}, w_{i; 2})\cdots R_{4n-1}(r_{i; 4n-1}, w_{i; 4n-1})
R_0(s_{i; 0}, j\omega_{i; 0})\cdots R_{4n-1}(s_{i; 4n-1}, j\omega_{i; 4n-1})\\
&=R_2(r'_{i; 1}, w'_{i; 2})\cdots R_{4n-1}(r'_{i; 4n-1}, w'_{i; 4n-1})
R_0(s'_{i; 0}, j\omega'_{i; 0})\cdots R_{4n-1}(s'_{i; 4n-1}, j\omega'_{i; 4n-1}).
\end{split}
\end{equation*}
This leads us to the result that $x_2=x'_2$ and 
$R_2(r_{i; 2}, w_{i; 2})=R_2(r'_{i; 2}, w'_{i; 2})$. Repeating this procedure successibly we can be led to $x_3=x'_3$, $\ldots$, $x_{4n-1}=x'_{4n-1}$ via inductive reduction, which prove $x=x'$. This also asserts that ($\ast$) can be reduced to
\begin{equation*}\tag{$\ast\ast$}
\begin{split}
R_0(s_{i; 0}, j\omega_{i; 0})\cdots &R_{4n-1}(s_{i; 4n-1}, j\omega_{i; 4n-1})\\
&=R_0(s'_{i; 0}, j\omega'_{i; 0})\cdots R_{4n-1}(s'_{i; 4n-1}, j\omega'_{i; 4n-1}).
\end{split}
\end{equation*}

By going through the process similar to when proving $x=x'$ above 
based on this equation we obtain $y=y'$ and thereby we can conclude that 
$(x, y)=(x', y')$. This completes the proof of the lemma.
\end{proof}

\section{Proof of  $\mathrm{(i)}$}

Let us put
\[D_l(\tau_{2l-1}, z\tau_{2l})=\mathrm{diag}(1, \overset{(2l-2)}\ldots, 1, \tau_{2l-1}, 
\bar{\tau}_{2l-1}z\tau_{2l}, \bar{z}\bar{\tau}_{2l}, 1, \ldots, 1)\qquad 
(1\le l\le 2n)\] 
where $\tau_{2l-1}, \tau_{2l}, z\in S^1$. 
Let $P_l$ be the subspace of $G$ consisting of 
$D_l(\tau_{2l-1}, z\tau_{2l})d(\bar{z})$. Then it forms the total space of a principal $S$-bundle over $T_l=S^1\times S^1$ along with the projection map of $p_l : P_l\to T_l$ given by $D_l(\tau_{2l-1}, z\tau_{2l})d(\bar{z})\to
 (\tau_{2l-1}, z\tau_{2l})$ where $T_l$ is considered as an 
subspace of $G/S$ under $\iota_l : (\tau_{2l-1}, z\tau_{2l})\to 
p(D_l(\tau_{2l-1}, z\tau_{2l}))$ which is clearly injective. 

Let $\tau_{2l-1}=e^{\eta i}$, $\tau_{2l}=e^{\theta i}$ for  
$0\le \eta, \, \theta < 2\pi$ and let $\mu_l : T_l\to S^2$ be the map given by 
\begin{equation*}
(e^{\eta i}, ze^{\theta i})\to\left\{
\begin{array}{ll}
(\cos(\eta/2),  ze^{\theta i}\sin(\eta/2))_R & \ (0\le \eta\le \pi)\\
(-\cos(\eta/2),  ze^{\theta t_\eta i}\sin(\eta/2))_R & \ (\pi\le \eta< 2\pi), \ \ \ 
t_\eta=2-\eta/\pi.
\end{array}
\right.
\end{equation*}
Then taking into account the fact that a principal circle bundle over $S^1$ is 
trivial we see that the classifying map of $p_l$ factors through $S^2$ where the restriction of $p_l$ to $\{1\}\times S^1\subset T_l$ is viewed as being trivial. 
Therefore we have 

\begin{lem}[cf.\, ~\cite{LS}, \!\S2, Example 3] 
$p_l : P_l\to T_l$ is isomorphic to the induced bundle of the complex 
Hopf bundle $p : SU(2)\to S^2$ by $\mu_l$ and $\mu_l$ also induces an isomorphism 
$H^2(S^2, \mathbb{Z})\cong H^2(T_l, \mathbb{Z})$ for $1\le l\le 2n$.  
\end{lem}
\begin{proof}
In order to prove the first equation it suffices to show that there is 
a bundle map covering $\mu_l$. In fact we see that based on the above assumption  the assignment   
\begin{equation*}
D_l(e^{\eta i}, ze^{\theta i})d(\bar{z})\to\left\{
\begin{array}{ll}
R(\bar{z}\cos(\eta/2),  e^{\theta i}\sin(\eta/2)) & \ (0\le \eta\le \pi)\\
R(-\bar{z}\cos(\eta/2),  e^{\theta t_\eta i}\sin(\eta/2)) & \ (\pi\le\eta<2\pi), \ \ \ t_\eta=2-\eta/\pi
\end{array}
\right.
\end{equation*}
defines the desired bundle map $\tilde{\mu}_l : P_l\to SU(2)$. The second equation is immediate from the definition of $\mu_l$.
\end{proof}

Letting \[
D(\tau_{2l-1}, z_l\tau_{2l})=\textstyle\prod_{l=1}^{2n}
D_l(\tau_{2l-1}, z_l\tau_{2l})\quad\text{and}\quad 
T^{2n}=\textstyle\prod_{l=1}^{2n}T_l
\] 
we define $\iota_{2n} : T^{2n}\to G/S$ to be the map given by 
$\lambda=(\lambda_1, \ldots, \lambda_{2n})\to 
p(D(\tau_{2l-1}, z_l\tau_{2l}))$ where $\lambda_l=(\tau_{2l-1}, z_l\tau_{2l})$. Then  taking into account the analysis performed for the proof of Corollary 2.2 we have 
from Lemma 3.1
\begin{equation}
\iota_{2n}^*E\cong L^{\boxtimes{(2n)}}.
\end{equation}
 
Furthermore, put 
\[D_l(s_l, jv_l)=\mathrm{diag}(1, \overset{(l-1)}\ldots, 1, s_l+jv_l, 1, \ldots, 1)\qquad (1\le l\le 4n)\] 
where $s_l+jv_l\in Sp(1)$ with $s_l\ge 0$, $v_l\in \mathbb{C}$. Similarly 
to the above case let $P_l$ be the subspace of $G$ consisting of 
$D_l(s_l, jv_l)d(\bar{z})$. Then it forms the total space of a principal $S$-bundle over $S^2$ along with the projection map of $p_l : P_l\to S^2$ given by 
$D_l(s_l, jv_l)d(\bar{z})\to (s_l,  jv_l)_R$. But in this case it is clear that  
$p_l : P_l\to S^2$ is isomorphic to $p : SU(2)\to S^2$ as a principal 
$S$-bundle over $S^2$. Let $\iota_{4n+1} : (S^2)^{4n+1}\to G/S$ be 
the map given by $\xi=(\xi_1, \ldots, \xi_{4n+1})\to p(D(s_l, jv_l))$ where  
$\xi_l=(s_l,  jv_l)_R$ and $D(s_l, jv_l)=\textstyle\prod_{l=1}^{4n+1}
D_l(s_l, jv_l)$. Then we have 
\begin{equation}
\iota_{4n+1}^*E\cong L^{\boxtimes{(4n+1)}}.
\end{equation} 

In the notation above, let 
\[\psi : (S^2)^{16n^2+4n}\times T^{2n}\times (S^2)^{4n+1}\to G/S\] be the map given by 
\[(x, \lambda, \xi, y)\to p(R(r_{i; k}, w_{i; k})D(\tau_{2n-1}, z_l\tau_{2l})
D(s_l, jv_l)R(s_{i; k}, j\omega_{i; k})).\]

Let $(T_l)^\circ=T_l-\{1\}\times S^1$ and 
$(T^{2n})^\circ=\textstyle\prod_{l=1}^{2n}(T_l)^\circ$. Then we have 
\begin{lem}
The restriction of $\psi$ to $((S^2)^{16n^2+4n})^\circ\times (T^{2n})^\circ
\times ((S^2)^{4n+1})^\circ$ is an injective map.
\end{lem}
\begin{proof}
Suppose $\psi(x, \lambda, \xi, y)=\psi(x', \lambda', \xi', y')$ 
in terms of the notations used in the proof of Lemma 2.3.
Then by definition it can be viewed as the equation 
\begin{equation*}\tag{$\ast$}
\begin{split}
R(r_{i; k}, w_{i; k})&D(\tau_{2l-1}, z_l\tau_{2l})D(s_l, jv_l)
R(s_{i; k}, j\omega_{i; k})\\
&=R(r'_{i; k}, w'_{i; k})D(\tau'_{2l-1}, z_l\tau'_{2l})
D(s'_l, jv'_l)R(s'_{i; k}, j\omega'_{i; k})
\end{split}
\end{equation*}
with $r_{i; k}, s_l, s_{i; k}>0$ and $r'_{i; k}, s'_l, s'_{i; k}>0$.

From this we want to derive $x=x'$, $\lambda=\lambda'$ 
$\xi=\xi'$ and $y=y'$; that is, $r_{i; k}=r'_{i; k}$, $w_{i; k}=w'_{i; k}$, 
$s_{i; k}=s'_{i; k}$, $\omega_{i; k}=\omega'_{i; k}$, $s_l=s'_l$, $v_l=v'_l$ 
and $\tau_{2l-1}=\tau'_{2l-1}$, $\tau_{2l}=\tau'_{2l}$. 
However under this assumption we find that if the proof is proceeded along the same lines as in the proof of the injectivity of $\phi$ in Lemma 2.3, then in order to prove this lemma it suffices to show that these equalities hold 
only in the case where $k=0$, $l=1$.

Now looking back the case $k=0$ in the proof of Lemma 2.3 and following the process mentioned there we know that with the notaion used there 
\[\tau_1(s_1+jv_1)\underline{a}_{l, 0; 0}
=\tau'_1(s'_1+jv'_1)\underline{a}'_{l, 0; 0}\qquad 
(1\le l\le 4n)\]
holds because $s_{1; 0}\cdots s_{4n; 0}>0$; that is, 
$\tau_1s_1\underline{a}_{l, 0; 0}=\tau'_1s'_1\underline{a}'_{l, 0; 0}$ and 
$\tau_1v_1\underline{a}_{l, 0; 0}=\tau'_1v'_1\underline{a}'_{l, 0; 0}$ 
$(1\le l\le 4n)$. 
First, using the former equation together with 
$|\underline{a}_{1, 0; 0}|^2+\cdots+|\underline{a}_{4n, 0; 0}|^2=1$, 
$|\underline{a}'_{1, 0; 0}|^2+\cdots+|\underline{a}'_{4n, 0; 0}|^2=1$
and $|\tau_1|=|\tau_1'|=1$ we have $s_1=s'_1$. Substituting this into the former with $l=0$ we have 
$\tau_1(r_{1; 0}\cdots r_{4n; 0})=\tau'_1(r'_{1; 0}\cdots r'_{4n; 0})$. 
Now since $r_{l; 0}r_{l; 0}'> 0$ for all $l$ we have $\tau_1=\tau'_1$.
Applying this to the latter equation we have $v_1=v'_1$ along the same lines. Thus we have  
\[\tau_1=\tau'_1, s_1=s'_1, v_1=v'_1 \quad \text{and} \quad 
\underline{a}_{l, 0; 0}=\underline{a}'_{l, 0; 0} \ \ \ (1\le l\le 4n-1).\]
Repeating this procedure we are led to  
\[(x, \lambda, \xi)=(x', \lambda', \xi')\]
and consequently the equation $(\ast)$ above can be reduced to 
\[R(s_{i; k}, j\omega_{i; k})=R(s'_{i; k}, j\omega'_{i; k}).\]
This equation is the same as $(\ast\ast)$ in the proof of Lemma 2.3 and so 
it of course leads us to the result $y=y'$, which completes the proof of the lemma.
\end{proof}
\begin{proof}[Proof of $\mathrm{(i)}$]
Put $B= (S^2)^{(16n^2+4n)}\times T^{2n}\times (S^2)^{4n+1}$. Then clearly 
$\dim B=\dim G/S$. From this fact and the injectivity result of $\psi$ given in Lemma 3.2 by taking into account the construction (cf.\!~\cite{H}) we have that 
$\psi$ can be continuously deformed into an onto map of degree one from $B$ to 
$G/S$. Let $[M]$ denote the fundamental class of a manifold $M$. Then  
$\psi_*([B])=[G/S]$, so we have
\[
\langle c_1(E)^{16n^2+10n+1}, \, [G/S]\rangle
=\langle (c_1(\psi_*E)^{16n^2+10n+1}), \, [B]\rangle. 
\]
Hence by Corollary 2.2, (3.1) and (3.2)  
\begin{equation*}
\begin{split} 
&\langle c_1(E)^{16n^2+10n+1}, \, [G/S]\rangle
=\langle c_1(L^{\boxtimes(16n^2+10n+1)}, \, [B] \rangle\\
&=\langle  c_1(L^{\boxtimes(16n^2+4n)}), \, [(S^2)^{(16n^2+4n)}] \rangle\,
\langle c_1(L^{\boxtimes(2n)}, \, [T^{2n}]\rangle\,
\langle  c_1(L^{\boxtimes(4n+1)}), \, [(S^2)^{(4n+1)}\rangle]
\end{split}
\end{equation*}
Substituting this into the equation of Proposition 2.1 of ~\cite{LS} we obtain
\begin{equation*}
e_{\mathbb C}([S(E), \Phi_E])=(-1)^nB_{8n^2+5n+1}/2(8n^2+5n+1)
\end{equation*}
where $\Phi_E$ denotes  the trivialization of the stable tangent space of $S(E)$ derived by the framing $\mathscr{L}_S$ on $G/S$ induced by $\mathscr{L}$ 
due to it being $S$-equivariant (cf.\!~\cite{K}, p.\,42; ~\cite{LS}, p.\,36). 

Here by overlaying $\Phi_E$ over $\mathscr{L}$ we find that it is trivialized in 
$2n$ parts on $G$. This is due to adding the terms 
$D(\tau_{2l-1}, z\tau_{2l})D(s_{l'}, jv_{l'})$ $(1\le l\le 2n)$ in the definition formula of 
$\psi$ to that of $\phi$. But since 
$D(\tau_{2l-1}, z\tau_{2l})=R_{1; 2l-1}(\tau_{2l-1}, 0)R_{1; 2l}(z\tau_{2l}, 0)$ we see that in order to remove this discrepancy it suffices to replace $\mathscr{L}$ by the framing 
$\mathscr{L}^{2n\rho}$ obtained by twisting $\mathscr{L}$ by $2n\rho$ where 
$\rho$ denotes the canonical real representation 
$Sp(4n+1)\to GL(16n+1, \mathbb{R})$. Thus we have the desired result. This 
completes the proof of (i).
\end{proof}

\section{Decomposition of $E$ in the spin case}

In order to prove (ii) we follow the lines of the proof of (i).

Let $\{e_1, \ldots, e_N\}$ be an orthonormal basis of $\mathbb{R}^N$, i.e. 
a basis subject to the relations $e_i^2=-1$ and $e_ie_j=-e_je_i$ for $i\ne j$. Let $Cl_N$ denote the Clifford algebra generated by $e_1, \ldots, e_N$ and 
$Cl_N^0\subset Cl_N$ the even part, the subalgebra generated by the linear combinations of the even degree elements.  Then $Spin(N)$ is defined to be the 
multiplicative subgroup of $Cl_N^0$ where $Cl_N^0$ is isomorphic to 
the even part $\Lambda^\mathrm{ev}(\mathbb{R}^N)$ of the exterior algebra 
of $\mathbb{R}^N$ as a vector space.

Let $G=Spin(8n-2)$. For $1\le l\le 4n-2$ we put
\[
z_{l, l+1}(t)=\cos t+e_{l}e_{l+1}\sin t \quad \text{and}\quad
d(t)=\textstyle\prod_{l=1}^{4n-2}z_{2l, 2l+1}(t)
\] 
where $t\in\mathbb{R}$. It is clear that $z_{2l, 2l+1}(t)$ behaves as an 
element of the unit sphere $S^1$. The relation 
$z_{2l, 2l+1}(t)z_{2l, 2l+1}(t')=z_{2l, 2l+1}(t+t')$ $(t'\in\mathbb{R})$ holds, so as usual 
we write $|z_{2l, 2l+1}(t)|=(z_{2l, 2l+1}(t)z_{2l, 2l+1}(-t))^{1/2}$. In addition, clearly the 
commutative property $z_{2l, 2l+1}(t)z_{2l', 2l'+1}(t)=z_{2l', 2l'+1}(t)z_{2l, 2l+1}(t)$ holds for $l\ne l'$, hence it follows that $d(t)$ generates a circle subgroup of $G$, 
denoted by $S$. 

Let $d(t)$ ($t\in\mathbb{R})$ act on $G$ by right multiplication. Let $p : G\to G/S$ be the principal bundle along with the natural projection. Let $\pi : E=G\times_S\mathbb{C}\to G/S$ be the canonical line bundle over $G/S$ associated to $p$ where $S$ acts on $\mathbb{C}$ as $S^1$, i.e. $d(t)v=e^{ti}v$ for $v\in \mathbb{C}$. Then  its unit sphere bundle $S(E) \to G/S$ is isomorphic to $p$ as a principal $S$-bundle over $G/S$. From now we prove that $E$ can be expressed as an exterior tensor product of complex line bundles.

Let $1\le m\le 4n-1$ such that $m\ne l+1$. Put  
\[
x_{2m-1, 2l, 2l+1}(r, s; t, \theta)=rz_{2l, 2l+1}(t)+e_{2m-1}e_{2l}\,(sz_{2l, 2l+1}(\theta))
\] 
where $r, s$ are $\ge 0$ and satisfy $r^2+s^2=1$ and $t, \theta\in\mathbb{R}$. 
Then it is clear that the set consisting of $x_{2m-1, 2l, 2l+1}(r, s; t, \theta)$ forms 
a subgroup of $G$ isomorphic to $Sp(1)$. From the above we see that the 
following commutative property holds
\begin{equation}
x_{2m-1, 2l, 2l+1}(r, s; t, \theta)d(t')
=d(t')x_{2m-1, 2l, 2l+1}(r, s; t, \epsilon\theta) 
\end{equation}
for $t, t'\in \mathbb{R}$ where $\epsilon=-1$ if $m\ge 2$ and $\epsilon=1$ 
if $m=1$. 

For $1\le k\le 4n-2$ we define $x_m(r, s; t, \theta)^{\{k\}}$ to be
\[
\bigl(\textstyle\prod_{l=1}^{4n-2}x_{2m-1, 2l, 2l+1}(r, s; t, \theta)\bigr) 
\ \text{with} \ r=1 \ \text{except} \ l=k; \ \text{namely} 
\]
\[
\bigl(\textstyle\prod_{l=1}^{k-1}z_{2l, 2l+1}(t)\bigr)
\bigl(r_kz_{2k, 2k+1}(t)+e_{2m-1}e_{2k}\,(s_kz_{2k, 2k+1}(\theta))\bigr)
\bigl(\textstyle\prod_{l=k+1}^{4n-2}z_{2l, 2l+1}(t)\bigr).
\]
Then 
\begin{equation}
\begin{split}
x_m(r, s; t, \theta)^{\{k\}}d(-t)&=r+e_{2m-1}e_{2k}\,(sz_{2k, 2k+1}(\theta-t))\\ 
&=x_{2m-1, 2k, 2k+1}(r, s; 0, \theta-t). 
\end{split}
\end{equation}
Let
\[
x_m(r_l, s_l; t_l, \theta_l)=\textstyle\prod_{l=1}^{4n-2}x_m(r_l, s_l; t_l, \theta_l)^{\{l\}}
\]
where $r_l, s_l\ge 0$ and $t_l, \theta_l\in \mathbb{R}$. 
Then using (4.2) together with (4.1) we have
\begin{equation}
x_m(r_l, s_l; t_l, \theta_l)d(-t_1)\cdots d(-t_{4n-2})
=\textstyle\prod_{l=1}^{4n-2}x_{2m-1, 2l, 2l+1}(r_l, s_l; 0, \epsilon\theta_l-t_l)
\end{equation}
where $\epsilon=(-1)^{l-1}$ if $m\ge 2$ and $\epsilon=1$ if $m=1$.
Here if $r_k=0$ for some $k$, then since $s_k=1$ we see that it is possible to replace both variables $t_k$ and $\theta_k$ by $t_k-\epsilon\theta_k$ in the $k$-th component $x_m(r_k, s_k; t_k, \theta_k)^{\{k\}}$ of $x_m(r_l, s_l; t_l, \theta_l)$. 
From this we find that $x_{2m-1, 2k, 2k+1}(r_k, s_k; 0, \epsilon\theta_k-t_k)$ on the right side can be transformed into $e_{2m-1}e_{2k}$ by replacing the variable $t_k$ of $d(-t_k)$ by $t_k-\epsilon\theta_k$. This means that when we regard 
$x_m(r_k, s_k; t_k, \theta_k)^{\{k\}}d(-t_k)$ with $r_k=0$ as an element of $G/S$, it can be converted to $e_{2m-1}e_{2k}$. 

For fixed $m, k$, the elements $x_m(r, s; t, \theta)^{\{k\}}$ generate an $S$-invariant subspace $P_m^{\{k\}}\subset G$ isomorphic to $Sp(1)=S^3$, as 
mentioned above.
From (4.2) we see that its quotient $P_m^{\{k\}}/S$ consists of 
$x_{2m-1, 2k, 2k+1}(r, s; 0, \theta)$, that is, of 
$r+e_{2m-1}e_{2k}\,(s_kz_{2k, 2k+1}(\theta))$ and then the assignment
$r+e_{2m-1}e_{2k}\,(s_kz_{2k, 2k+1}(\theta))\to (1-2r^2, 2rse^{\theta i})$ induces 
a homeomorphism of the quotient above onto the unit 2-sphere $S^2$ of 
$\mathbb{R}\times\mathbb{C}$. Below similarly to above we write $(r, w)_R$ for $(1-2r^2, 2rw)$ with $w=se^{\theta i}$ using the notation above. Then from the observation  above we know that the variable $w$ must be converted to 1 when $r=0$, that is, $(r, w)_R=(0, 1)$. 

Let $\iota_m^{\{k\}} : S^2\to G/S$ be the map given by 
$(r_k, w_k)_R\to r_k+e_{2m-1}e_{2k}\,\omega_k$ where we﻿ write $w_k=s_ke^{\theta_k i}$ and $\omega_k=s_kz_{2k, 2k+1}(\theta_k)$.  Then the induced bundle $\iota_m^{\{k\}*}p$ becomes isomorphic to the Hopf bundle $S^3\to S^2$ and therefore we have 
\begin{equation}
\iota_m^{\{k\}*}E\cong L.
\end{equation}
Here again using the same symbol $L$ we denote the line bundle associated to 
$\iota_m^{\{k\}*}p$. 

Let $(S^2)^l=S^2\times\overset{(l)}\cdots\times S^2$ and 
$\phi_m : (S^2)^{4n-2}\to G/S$ be the map given by $x=(x_1, \ldots, x_{4n-2})\to 
p\bigl(\prod_{k=1}^{4n-2}(r_k+e_{2m-1}e_{2k}\omega_k)\bigr)$ where 
$x_k=(r_k, w_k)_R\in S^2$. 
\begin{lem} 
The restriction of $\phi_m$ to $((S^2)^{4n-2})^\circ$ is an injective map.
\end{lem}
\begin{proof}
Suppose $\phi_m(x)=\phi_m(x')$, where $x'=(x'_1, \ldots, x'_{4n-2})$ and 
$x'_k=(r'_k, w'_k)_e$. By definition this equation can be regarded as meaning
\[
\textstyle\prod_{k=1}^{4n-2}(r_k+e_{2m-1}e_{2k}\omega_k)
=\textstyle\prod_{k=1}^{4n-2}(r'_k+e_{2m-1}e_{2k}\omega'_k)
\]
with $r_k>0$ and $r'_k>0$. Using this we proceed by induction on $k$. Since
$r_i\ne 0$, i.e. $r_i>0$ for all $i$ we first have 
\[
r_1\cdots r_{4n-2}=r'_1\cdots r'_{4n-2} \quad \text{and} \quad 
r_2\cdots r_{4n-2}\omega_1=r'_2\cdots r'_{4n-2}\omega'_1.
\]
This together with $r_1^2+|\omega_1|^2={r'}_1^2+|\omega'_1|^2=1$ gives 
$(r_2\cdots r_{4n-2})^2=({r'}_2\cdots {r'}_{4n-2})^2$, so $r_1^2={r'}_1^2$, hence we have $r_1={r'}_1$, $r_2\cdots r_{4n-2}=r'_2\cdots r'_{4n-2}$ and $w_1=w'_1$. This shows that $x_1=x'_1$, and in addition allows us to replace the equation above by 
\[
\textstyle\prod_{k=2}^{4n-2}(r_k+e_{2m-1}e_{2k}\omega_k)
=\textstyle\prod_{k=2}^{4n-2}(r'_k+e_{2m-1}e_{2k}\omega'_k)
\]
with the same relation as given above. This tells us that by successively 
repeating the above procedue we can be led to the remaning results 
$x_2=x'_2, \ldots, x_{4n-2}=x'_{4n-2}$. This proves $x=x'$ and therefore the 
lemma.
\end{proof}

Let $\phi : (S^2)^{16n^2-12n+2} \to G/S$ be the map given by 
\[
x=(x_1, \ldots, x_{4n-1})\to 
p\bigl(\textstyle\prod_{m=1}^{4n-1}
\bigl(\textstyle\prod_{k=1}^{4n-2}(r_{k, m}+e_{2m-1}e_{2k}\omega_{k, m})\bigr)\bigr)
\] 
where $x_m=(x_{1, m}, \ldots, x_{4n-2, m})$ and 
$x_{k, m}=(r_{k, m}, w_{k, m})_R\in S^2$. 

\begin{lem}
The restriction of $\phi$ to $((S^2)^{16n^2-12n+2})^\circ$ is an injective map.
\end{lem}
\begin{proof}
The proof is quite similar to the previous lemma.
\end{proof}
From this taking into account (4.3) and the isomorphism of (4.4) we obtain  
\begin{cor}
$\phi^*E\cong L^{\boxtimes(16n^2-12n+2)}$.
\end{cor}

\section{Proof of  $\mathrm{(ii)}$}

Put
\begin{equation*}
\begin{split}
D_l(\eta_l, t_l, \theta_l)&=z_{4l-2, 4l}(\eta_l)z_{4l-2, 4l-1}(t_l)
z_{4l-2, 4l-1}(\theta_l)\\
&=z_{4l-2, 4l}(\eta_l)z_{4l-2, 4l-1}(t_l+\theta_l) \qquad (1\le l\le 2n-1)
\end{split}
\end{equation*}
where $\eta_l t_l, \theta_l\in\mathbb{R}$. Then similarly to (4.1) 
we see that the commutative property    
\begin{equation}
D_l(\eta_l, t_l, \theta_l)d(t')=d(t')D_l(\eta_l, t_l, \theta_l)
\quad (t'\in \mathbb{R})
\end{equation}
holds. Let $P_l\subset G$ be the subspace consisting of 
$D_k(\eta_l, t_l, \theta_k)d(-t_l)$ $(\eta_l, t_l, \theta_l\in\mathbb{R})$. 
Then it follows that it forms the total space of a principal $S$-bundle over $T_l=S^1\times S^1$ along with the projection map of $p_l : P_k\to T_l$ given by $D_l(\eta_l, t_l \theta_l)d(-t_l)\to (e^{\eta_qi}, e^{(t_l+\theta_l)i})$ where $T_l$ can be considered as a subspace of $G/S$ under 
$\iota_l : (e^{\eta_qi}, e^{(t_l+\theta_l)i})\to p(D_l(\eta_l, t_l, \theta_l))$. 

Let $\mu_l : T_l\to S^2$ be the map given by 
\begin{equation*}
(e^{\eta_li}, e^{(t_l+\theta_l)i}) \to \left\{
\begin{array}{ll}
(\cos(\eta_l/2),  e^{(t_l+\theta_l)i}\sin(\eta_l/2))_R & \ (0\le \eta_l\le \pi)\\
(-\cos(\eta_l/2),  e^{t_l\theta_lt_{\eta_l}i}\sin(\eta_l/2))_R & \ (\pi\le \eta_l< 2\pi),\quad 
t_{\eta_l}=2-\eta_l/\pi.
\end{array}
\right.
\end{equation*}
Taking into account the fact that a principal circle bundle over $S^1$ is trivial we also see that the classifying map of $p_l$ factors through $S^2$ where  
the restriction of $p_l$ to $\{1\}\times S^1\subset T_l$ is viewd as being trivial. Then similarly to Lemma 3.1 we have 
\begin{lem}[cf.\, ~\cite{LS}, \!\S2, Example 3] $p_l : P_l\to T_l$ is isomorphic to the induced bundle of the complex 
Hopf bundle $p : SU(2)\to S^2$ by $\mu_l$ and $\mu_l$ also induces an isomorphism $H^2(S^2, \mathbb{Z})\cong H^2(T_l, \mathbb{Z})$ for 
$1\le l\le 2n-1$.  
\end{lem}
\begin{proof}
For the same reason as in the symplectic case above we have a bundle map 
$\tilde{\mu}_l : P_l\to SU(2)$ covering $\mu_l$ which can be induced by  
\begin{equation*}
D_l(\eta_l, t_l, \theta_l)d(-t_l)\to \left\{
\begin{array}{ll}
M(e^{-t_li}\cos(\eta_l/2),  e^{\theta_l i}\sin(\eta_l/2)) & \ (0\le \eta_l\le \pi)\\
M(-e^{-t_li}\cos(\eta_l/2),  e^{\theta_lt_{\eta_l} i}\sin(\eta_l/2)) & \ (\pi\le\eta_l<2\pi)
\end{array}
\right.
\end{equation*}
where $t_{\eta_l}=2-\eta_l/\pi$ and $M(a, b)$ is as above. This shows that 
the first equation holds true. The second equation is immediate from the definition of $\mu_l$.
\end{proof}

Let $D(\eta_l, t_l, \theta_l)=\textstyle\prod_{l=1}^{2n-1}
D_l(\eta_l, t_l, \theta_l)$ and $T^{2n-1}=\textstyle\prod_{l=1}^{2n-1}T_l$. 
Let $\iota_{2n-1} : T^{2n-1}\to G/S$ be the map given by 
$\lambda=(\lambda_1, \ldots, \lambda_{2n-1})\to 
p(D(\eta_l, t_l, \theta_l))$ where 
$\lambda_l=(e^{\eta_li}, e^{\tau_l i})$ with $t_l+\theta_l$ replaced by $\tau_l$ 
for simplicity. 
\begin{lem}
$\iota_{2n-1}$ is an injective map.
\end{lem}
\begin{proof}
Suppose $\iota_{2n-1}(\lambda)=\iota_{2n-1}(\lambda')$ where 
$\lambda'=(\lambda'_1, \ldots,\lambda'_{2n-1})$ and 
$\lambda'_l=(e^{\eta'_li}, e^{\tau'_l i})$. 
This can be viewed as meaning 
\[\textstyle\prod_{l=1}^{2n-1}z_{4l-2, 4l}(\eta_l)z_{4l-2, 4l-1}
(\tau_l)
=\textstyle\prod_{l=1}^{2n-1}z_{4l-2, 4l}(\eta'_l)z_{4l-2, 4l-1}
(\tau'_l).\]
Multiplying by $z_{2, 4}(-\eta_1)$ from the left we have 
\begin{equation*}
\begin{split}
z_{2, 3}(\tau_1)&\textstyle\prod_{l=2}^{2n-1}z_{4l-2, 4l}(\eta_l)
z_{4l-2, 4l-1}(\tau_l)\\
&=z_{2, 4}(\eta'_1-\eta_1)z_{2, 3}(\tau'_1)
\textstyle\prod_{l=2}^{2n-1}z_{4l-2, 4l}(\eta'_l)z_{4l-2, 4l-1}(\tau'_l).
\end{split}
\end{equation*}
Then by definition it must be that $\sin(\eta'_1-\eta_1)=0$, so 
$\eta\equiv\eta'$ $(\!\!\!\!\mod 2\pi)$. This implies that 
$z_{2, 4}(\tau_1)=z_{2, 4}(\tau')$ and therefore the equation above can be transformed into 
\begin{equation*}
z_{2, 3}(\tau_1)\textstyle\prod_{l=2}^{2n-1}z_{4l-2, 4l}(\tau_l)
z_{4l-2, 4l-1}(\theta_l)=z_{2, 3}(\tau'_1)
\textstyle\prod_{l=2}^{2n-1}z_{4l-2, 4l}(\eta'_l)z_{4l-2, 4l-1}(\tau'_l).
\end{equation*}
Multiplying by $z_{2, 3}(-\tau_1)$ from the left similarly we obtain 
$\tau_1=\tau'_1$ 
$(\!\!\!\!\mod 2\pi)$ and hence $\lambda_1=\lambda'_1$. This also allows us to rewrite the above equality as 
\begin{equation*}
\textstyle\prod_{l=2}^{2n-1}z_{4l-2, 4l}(\eta_l)z_{4l-2, 4l-1}(\tau_l)
=\textstyle\prod_{l=2}^{2n-1}z_{4l-2, 4l}(\eta'_q)z_{4l-2, 4l-1}(\tau'_l).
\end{equation*}
From this we see that applying the process above to the cases $l=2, \ldots, 
2n-1$ successively we can gets the results $\lambda_l=\lambda'_l$  
$(2\le l\le 2n-1)$ which leads us the 
conclusion that $\lambda=\lambda'$. Thus the lemma is proved.
\end{proof}
From Lemmas 5.1 and 5.2, due to (5.1), we have
\begin{cor} $\iota_{2n-1}^*E\cong L^{\boxtimes{(2n-1)}}$.
\end{cor}

Let $\psi : T^{2n-1}\times (S^2)^{16n^2-12n+2} \to G/S$ be the map given by 
\[(\lambda, x) \to 
p\bigl(D(\eta_l, t_l, \theta_l)\textstyle\prod_{m=1}^{4n-1}
\bigl(\textstyle\prod_{k=1}^{4n-2}(r_{k, m}+e_{2m-1}e_{2k}w_{k, m})\bigr)\bigr)
\] 
where $\lambda, x$ are as above. 

\begin{lem}
The restriction of $\psi$ to $(T^{2n-1})^\circ\times ((S^2)^{16n^2-12n+2})^\circ$  is an injective map.
\end{lem}
\begin{proof}
This can be done in a similar way to Lemmas 4.1, 4.2 and 5.2. 
Suppose $\psi(\lambda, x)=\psi(\lambda', x')$ in the notation above. 
Similarly to the above, we also think of this equation as 
\begin{equation*}
\begin{split}
\bigl(\textstyle\prod_{l=1}^{2n-1}&z_{4l-2, 4l}(\eta_l)z_{4l-2, 4l-1}
(\tau_l)\bigl)\,\bigl(\textstyle\prod_{m=1}^{4n-1}
\bigl(\textstyle\prod_{k=1}^{4n-2}(r_{k, m}+e_{2m-1}e_{2k}w_{k, m})\bigr)\bigr)
\\
&=\bigl(\textstyle\prod_{l=1}^{2n-1}z_{4l-2, 4l}(\eta'_l)z_{4l-2, 4l-1}
(\tau'_l)\bigr)\,\bigl(\textstyle\prod_{m=1}^{4n-1}
\bigl(\textstyle\prod_{k=1}^{4n-2}(r'_{k, m}+e_{2m-1}e_{2k}w'_{k, m})\bigr)\bigr)
\end{split}
\end{equation*}
where in this case $r_{k, m}>0$ and, $r'_{k, m}>0$. The product elements on both sides can be devided into two parts, the formers of which are the product elements with respect to the variables $\lambda$ and $\lambda'$, and the latters are the product elements with respect to the variables $x$ and $x'$, respectively. 

First we apply the method of proof of Lemma 4.2 to this equation taking into account the composition of the latter product elements, and thereby obtain
$\lambda=\lambda'$. Owing to this result the above equation can be also transformed into 
\begin{equation*}
\textstyle\prod_{m=1}^{4n-1}
\bigl(\textstyle\prod_{k=1}^{4n-2}(r_{k, m}+e_{2m-1}e_{2k}w_{k, m})\bigr)
=\textstyle\prod_{m=1}^{4n-1}
\bigl(\textstyle\prod_{k=1}^{4n-2}(r'_{k, m}+e_{2m-1}e_{2k}w'_{k, m})\bigr).
\end{equation*}
Using this, as seen in the proofs of Lemmas 4.1, 4.2, we can obtain
$x=x'$, thereby achieving $(\lambda, x)=(\lambda', x')$. This 
proves the lemma.
\end{proof}
\begin{proof}[Proof of $\mathrm{(ii)}$] 
The proof is similar to that of (i) and is proceeded as follows. 
Put $B=T^{2n-1}\times (S^2)^{16n^2-12n+2}$. Then $\dim B=\dim G/S$. From 
this and Lemma 5.4, due to the construction, it follows that $\psi$ is an onto map of degree one. So as in the proof of the case (i) we have $\psi_*([B])=[G/S]$ and therefore 
\begin{equation*}
\begin{split} 
&\langle c_1(E)^{16n^2-10n+1}, \, [G/S]\rangle
=\langle (c_1(\psi_*E)^{16n^2-10n+1}), \, [B]\rangle. \\
&=\langle c_1(L^{\boxtimes(2n-1)}, \, [T^{2n-1}]\rangle\,
\langle  c_1(L^{\boxtimes(16n^2-12n+2)}), \, [(S^2)^{(16n^2-12n+2)}\rangle.
\end{split}
\end{equation*}
Substituting this into the equation of Proposition 2.1 of ~\cite{LS} we obtain
\begin{equation*}
e_{\mathbb C}([S(E), \Phi_E])=(-1)^nB_{8n^2-5n+1}/2(8n^2-5n+1)
\end{equation*}
where $\Phi_E$ denotes the trivialization of the stable tangent space 
of $S(E)$, i.e. of $G$ derived from the induced framing on $G/S$. 

Here when considering the induced framing on 
$G/S$ we have $2n-1$ parts on $G$ on which 
$\mathscr{L}$ must be interpreted as being trivial. This is due to adding the terms $D_l(\eta_l, t_l, \theta_l)$ $(1\le l\le 2n-1)$ to those of the map $\phi$. Let $\Delta$ denote the spin representation of $Spin(8n-2)$. Then we see that this discrepancy can be removed by replacing 
$\mathscr{L}$ with the framing obtained by twisting by $(2n-1)\Delta$. 
This proves (ii) and therefore the theorem is proven.
\end{proof}

\begin{remark} From the proof of the theorem we see that
\[e_\mathbb{C}([Sp(4n+1), \mathscr{L}])=0,\qquad 
e_\mathbb{C}([Spin(8n-2), \mathscr{L}])=0.\]
For example, in the case of $Spin(8n-2)$, as seen just above, the doubling of the framing occurred there can be dissolved by thinking of the restriction of $E$ to $T_l$ for every $l$ as a trivial complex line bundle. But instead its first Chern class becomes zero and so, according to Proposition 2.1 of ~\cite{LS}, the value of $e_\mathbb{C}$ must become zero. The same is true of the case of $Sp(4n+1)$.
\end{remark}

\end{document}